\documentclass[12pt]{article}
\usepackage[utf8]{inputenc}
\usepackage{xcolor}
\usepackage{graphicx}
\usepackage{amsmath}
\usepackage{amssymb}
\usepackage{amsthm}

\newtheorem{Theorem}{Theorem}[section]
\newtheorem{Proposition}[Theorem]{Proposition}
\newtheorem{Definition}[Theorem]{Definition}
\newtheorem{Corollary}[Theorem]{Corollary}

\title{Elliptic Flowers: simply connected billiard tables with chaotic or non-chaotic flows moving around chaotic or non-chaotic cores}
\author{Leonid Bunimovich}
\date{March 2021}

\begin{document}

\maketitle

\begin{abstract}
We introduce a class of billiards with chaotic unidirectional flows (or non-chaotic unidirectional flows with ``vortices") which go around a chaotic or non-chaotic ``core", where orbits can change their orientation. Moreover, the corresponding billiard tables are simply connected in difference with many attempts to build billiards with interesting and/or exotic dynamics by putting inside billiard tables various ``scatterers" with funny shapes. Therefore the billiards in this new class are amenable to experimental studies in physics labs as well as to the rigorous mathematical ones, which may shed a new light on understanding of classical and quantum dynamics of Hamiltonian systems.
\end{abstract}

\section{Introduction}
Among the main advances of the 20th century science was the discovery that deterministic (dynamical) systems can demonstrate an enormously broad variety of behaviors, from completely regular (integrable) to completely chaotic. Most of the essential steps in this striking progress were based on finding various classes of dynamical systems, which demonstrated new, often unexpected and counter-intuitive, types of behavior. Billiards play a special role in the list of basic classes of dynamical systems. While being arguably the most ancient class of dynamical systems ever studied (what is called now integrability of billiards in circles, was known for millennia) it remains on a forefront of the theory of dynamical systems and their applications. This special role of billiards is mainly due to two factors. Billiards form one of the most visual classes of dynamical systems (if not just the most visual one) and besides, billiards are natural models of many real life phenomena, first of all in physics. 
The purpose of this paper is to introduce a new class of billiards, which demonstrate various new surprising properties. At the same time, these billiards are amenable to rigorous mathematical, as well as to promising numerical, and even to experimental, studies.
Formally these billiards could be considered as generalizations of the first class of billiards, which was presented as a justification of the discovery of a new mechanism of chaos (hyperbolicity) in dynamical systems \cite{3,4,5,6}, called the mechanism of defocusing. Boundaries of the first billiard tables, which demonstrated the mechanism of defocusing, consist of arcs of circles, so that the arcs completing all components of the boundary to the full circles lie within the corresponding billiard table. Such systems were called flower-like billiards \cite{5}, and later referred to as Bunimovich flowers (see e.g. \cite{2,20}). In what follows we will refer to this class of billiards as to circular flowers. It is worthwhile to mention that so beloved, especially by the physics community, stadium billiard is, in fact, a very special and even a singular example in this class. Indeed, a stadium arose from a flower with two circular petals by taking their common tangents. Such billiards were initially called squashes. A squash becomes stadium, when two circular arcs have the same radius.

The boundaries of the presented in this paper class of billiards consist of pieces of ellipses. Therefore these billiards will be called elliptic flowers (EF). However, in difference with circular flowers, which always have chaotic and ergodic dynamics, this class of billiards demonstrates a large variety of possible behaviors. For instance, a billiard is called a track billiard \cite{11} if almost all its orbits always move clockwise or counter-clockwise. Again, an ancient example of such dynamics is given by an integrable billiard between two concentric circles. A natural (although somewhat ambitious) question was whether there exist track billiards with chaotic behavior. This question arises in quantum chaos \cite{14,17} in search of examples/situations where Shnirelman peak \cite{19} in the spectrum of the Schroedinger operator may appear. In a track billiard this peak naturally appears in the phase space at the common boundary of tracks. Although existence of chaotic tracks sounded a little bit too exotic, it was proved that there exist such billiards \cite{11}. In elliptic flowers, which the present paper introduces, always exist tracks, which could be chaotic or non-chaotic. However, these tracks do not occupy the entire phase space, as it was in \cite{11}. Instead these tracks are going around a core, which also is an invariant set with a positive measure. Dynamics in the core also (as in tracks) can be chaotic or non-chaotic. A simple description of dynamics in the elliptic flowers billiards is the existence of a core surrounded by two flows (tracks) going in the opposite directions, which could be chaotic or have internal ``vortices". Dynamics in the core could be chaotic or be the one typical for a generic Hamiltonian system, i.e. a mixture of elliptic islands and chaotic ``seas". It should be noted that an integrable billiard in ellipse has two integrable tracks, consisting of orbits tangent to the confocal smaller ellipses. The corresponding (integrable) core consists of orbits which intersect the segment between the foci after any reflection off the boundary. Therefore the new in spirit elliptic flowers are the ones where either tracks or cores are chaotic, or the both are chaotic. Of course, if a billiard has a caustic, then orbits which do not intersect the caustic form two tracks, and the rest form a core. However, the issue is to investigate dynamics in tracks and core. Examples with regular dynamics, e.g. ellipses are well known.   

An important property of the elliptic flowers is that these billiard tables are simply connected. In other words, they do not have any internal boundaries, which are often used in constructions of billiards with various exotic behaviors. Therefore, elliptic flowers can be built as experimental devices joining in the labs mushrooms \cite{10}, which actually have more sophisticated, than the elliptic flowers, boundaries. In this respect, it is also worthwhile to mention that by varying the natural parameters, which describe the EFs, one can change dynamics of elliptic flowers in several ways. Such studies may shed a new light on our understanding of dynamics of generic Hamiltonian systems, both in the classical and quantum settings.

Another possible use of elliptic flowers comes from the theory of the slow-fast systems, where a recent breakthrough established importance of non-ergodicity for effective averaging over fast dynamics \cite{16,18}. Particularly, mushroom billiards, which always have a divided into KAM islands and ergodic chaotic ``seas" phase space, were effectively used in these studies \cite{16}. Indeed, the elliptic flowers, which we consider here, are always non-ergodic because of existence of at least three ergodic components, which are tracks and a core. Moreover, the boundaries of elliptic flowers can be smoothed, what cannot be done for mushroom billiards and for circular flowers. 

In the present paper the examples are given of elliptic flowers billiards with chaotic and non-chaotic tracks, surrounding chaotic and non-chaotic cores. Moreover, the first example of another phenomenon is given, which goes against the developed community intuition about billiards. It is universally known that, although mechanism of defocusing generates chaotic (hyperbolic) dynamics, the correlations in billiards with at least one focusing component of the boundary always decay slowly (power-like). On another hand, the correlations decay exponentially in dispersing billiards (\cite{1,22}). We demonstrate that in the chaotic cores of the elliptic flowers the correlations may decay exponentially. This result once again shows that the mechanism of dispersing is a special case of the mechanism of defocusing, not just formally, what is obvious, but also factually. Indeed, the purely focusing billiards can generate any type of behavior which dispersing billiards may generate as well as quite different dynamics from the one of dispersing billiards. (Defocusing is clearly more general because this mechanism allows the orbits to converge and diverge in configuration space, while dispersing allows only divergence of the orbits. Therefore dispersing is the limiting case of defocusing when orbits never converge).

Our goal here is to present the simplest sub-classes of the elliptic flowers billiards which demonstrate all the described above new types of dynamics. Therefore in this paper we consider only the simplest elliptic flowers billiards. It allows to provide either simple and visual proofs, or the results immediately follow from the already existing ones in the mathematical theory of billiards. These general results actually have quite sophisticated and long proofs, but they allow almost immediately conclude our results here. Already consideration of these simplest cases allows to get all the promised results on coexistence of chaotic or non-chaotic tracks with chaotic or non-chaotic cores. All four possible types of coexistence are demonstrated. 

However,  the dynamics of general elliptic flower billiards is much richer.  Therefore some natural questions, which arise for the future studies of the EFs billiards, are discussed in the last section. Some conjectures on possible existence of elliptic flowers with even more rich dynamics are also formulated there.  

\section{Multilayered Elliptic Flowers Billiards}

Construction of a multilayered elliptic flower billiard (MEF-billiard) starts with a choice of any convex polygon on the Euclidean plane. Let $A_1,A_2,...A_n$ be the vertices of a convex polygon $A$. This polygon will be called a base for the corresponding family of elliptic flowers billiards. The first elliptic layer over $A$ is formed by all ellipses, which do not intersect $A$ and have focuses  at the points $A_i$ and $A_{i+1}$, where $i$ varies between $1$ and $n$, and $i+1$ is taken modulus $n$. In other words, the first layer of ellipses over a polygon $A$ consists of all ellipses with focuses at the ends of one and the same side of $A$. The second layer of ellipses over $A$ is a union of all ellipses which have focuses at the points $A_i$ and $A_{i+2}$, where again $i+2$ is taken mod($n$). Clearly the second layer of ellipses over the polygon $A$ consists of all ellipses with focuses at the ends of small diagonals of $A$, which connect the ends of two neighboring sides of $A$. Analogously, a layer of ellipses number $m$ over a given polygon is defined in the same way, where, instead of $(i+2)$ (mod $n$), one should take $(i+m)$ (mod $n$). Therefore all ellipses have focuses at some vertices of $A$. It is easy to see that a regular polygon with $n$ vertices has $n/2$ layers, if $n$ is an even number, and $(n+1)/2$ layers, if $n$ is odd.

\begin{Definition}\label{def1}
A billiard table is called an unstructured elliptic flower over a base polygon $A$ if its entire boundary consists of pieces of ellipses which belong to layers of ellipses over $A$.
\end{Definition}

\begin{figure}[h]
	\begin{center}
		\includegraphics[width=6cm]{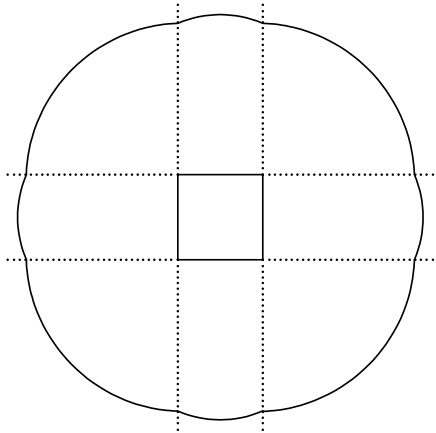}
		\caption{An elliptic flower billiard table there the base polygon is a square}~\label{fig1}
	\end{center}
\end{figure}
For example, in Fig. \ref{fig1} the base polygon $A$ is a square. There are only two layers in this case. Observe because the lines passing through the sides of $A$ never intersect outside of the polygon $A$. It is easy to see that such situation occurs only when $A$ is a triangle or a square.

We will consider in what follows a special subset of the set of all unstructured elliptic flowers, which are constructed in the following way. 

\begin{Definition}\label{def2}
Consider a base polygon $A$ with the vertices $A_1,A_2,...,A_n$. A billiard table $Q$ is called a structural elliptic flower over $A$ if all regular components of the boundary $\partial Q$ are the arcs of the ellipses from the layers over $A$, and each arc does not intersect the straight lines containing the sides of $A$.
\end{Definition}

\noindent{\bf Remark 1.} It follows from the Definition \ref{def2} that the ends of the elliptic arcs which form the boundary $\partial Q$ belong to the straight lines passing through the sides of the base polygon $A$ (see Fig. \ref{fig1}).

\noindent{\bf Remark 2.} Observe that all boundary components of a structural elliptic flower can belong only to one layer of ellipses over the base polygon $A$ if their ends are the points of intersection of the lines passing through the sides of the base polygon $A$. Also there are no structural elliptic flowers with all boundary components belonging to the first layer. (Fig. \ref{fig2} (a),(b))

\begin{figure}[h]
	\begin{center}
		\includegraphics[width=12cm]{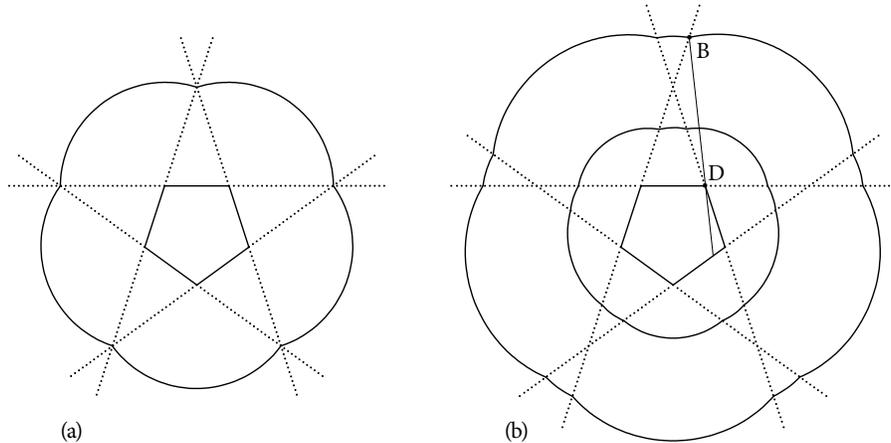}
		\caption{Structural elliptic flowers with a pentagon as a base polygon. (a) All arcs of ellipses belong to the second layer. (b) Regular components of the boundary belong to the 1st and 2nd layers in the smaller flower. In the larger flower petals belong to the 2nd and 3rd layers. The lines like BD cut out from the pentagon a smaller core. }~\label{fig2}
	\end{center}
\end{figure}

Clearly, the ends of any regular (smooth) component of the boundary of a structural elliptic flower belong to lines containing two different sides of the base polygon $A$ (Fig. \ref{fig1}, Fig. \ref{fig2}). However, if $A$ is a polygon with $n \geq 5$ sides then the lines containing the sides of $A$ do intersect outside of this base polygon (Fig. \ref{fig3}). As a result, these lines form a finite partition the plane into some compact closed sets and some infinite sets which we will call zones. 

It is worthwhile to mention that our construction is much more general than a  construction of a billiard table which contains a given caustic $K$. Recall that a curve $K$ is a caustic for a billiard if from one link of a billiard orbit being tangent to $K$ follows that all other links of this orbit are also tangent to $K$.    
Given a convex curve $K$, one can obtain a billiard table possessing this curve $K$ as a caustic by the following string construction. Take a rope with the length larger that the length of $K$. Then go around $K$ tightly keeping the rope at its full length. The resulting closed curve is a billiard table $Q$ which has $K$ as a caustics. This construction smoothens, i.e. even if the curve $K$ is only piecewise $C^1$ the so constructed curve (boundary of a billiard table $Q$) always will be globally $C^1$. For instance, if $K$ is a triangle then the boundary of the resulting  billiard table, consisting of six elliptic arcs, is globally $C^1$, but the curvature is discontinuous in the six points where the ellipses are glued together. 

\begin{figure}[h]
	\begin{center}
		\includegraphics[width=6cm]{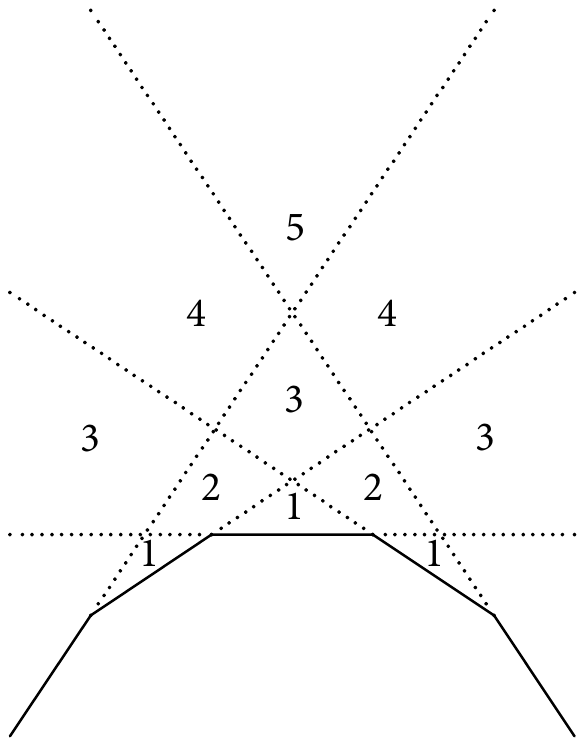}
		\caption{Partition of the plane into zones.}~\label{fig3}
	\end{center}
\end{figure}
For the sake of simplicity and clarity we will ignore a detailed consideration of elliptic flowers which can be constructed with a special procedure (choice of ellipses and their focuses) for each zone. Also for the same reason we will study in this paper only the simplest classes of the elliptic flower billiards. Therefore in what follows we will consider elliptic flowers billiards with the boundary components only from the first and/or the second layers. Moreover, we will study here only regular polygons with $n \leq 6$ as the bases of the EFs billiards. It is easy to see that in these cases there are only two layers and two zones if $n=3$ or $4$, and two layers and three zones if $n=5$ or $6$ (see Fig. \ref{fig2}). Remarkably, already these the simplest elliptic flowers have all the main new features which were discussed in the introduction. Besides, clearly only such elliptic flowers would be the most easy to build in the experimental labs.

Recall that a billiard is a dynamical system generated by an uniform motion of a point particle within a region $Q$ in the Euclidean plane. This region $Q$ is a configuration space of a billiard, which is called a billiard table. The boundary $\partial Q$ of $Q$ consists of a finite number of smooth (at least $C^2$) curves, which are called the regular components of $\partial Q$. A point particle moves along the straight lines within $Q$, and it gets elastically reflected off the boundary, i.e. the angle of incidence equals the angle of reflection. These angles are formed by the internal unit normal vector $n(q)$ at the point of reflection $q\in\partial Q$ and the vectors of the particle's velocities just before and just after a moment of the reflection off the boundary of the billiard table. Therefore, the orbits of billiards are broken lines, consisting of straight segments (links) which connect two points at the boundary of a billiard table, where two consecutive reflections occur. Without any loss of generality we may assume that a constant speed of the particle equals one.

Billiards are Hamiltonian systems. Therefore the billiard dynamics preserves the phase volume $\mu$. It is easy to see that the set of orbits which hit singular points of the boundary of a billiard table has measure zero. The major questions about billiards are related to ergodic and statistical properties of these dynamical systems with respect to the invariant measure $\mu$.

\begin{Definition}\label{def3}
A positive measure invariant set $Tr$ of a billiard dynamical system is called track if all orbits in $Tr$ move either clockwise or counter-clockwise in a billiard table.
\end{Definition}

Certainly, only positive measure tracks are of interest. For instance, a billiard within any convex region $Q$ with smooth boundary has zero measure ``tracks" consisting of tangent to $\partial Q$ orbits.

It is well-known that a billiard in ellipse $E$ is an integrable dynamical system. One positive measure family of orbits of such billiard consists of all orbits which are tangent between any two consecutive reflections off $E$ to some confocal with $E$ ellipse. Another positive measure family of orbits consists of all such orbits that the straight line, which contain a segment between any two consecutive reflections off $\partial E$, is tangent to some confocal with $E$ hyperbola. The two families of orbits, tangent to the confocal ellipses and hyperbolas, are separated in the phase space by the set of orbits which pass through the focuses of $E$ between any two consecutive reflections off the boundary. It is worthwhile to recall that, according to the optical definition of ellipses, an ellipse $E$ is a such closed curve that there are two points $F_1$ and $F_2$ with the property that a ray of light passing through $F_1$ will pass after reflection off $\partial E$ through $F_2$ , and vice versa. Certainly the orbits tangent to confocal ellipses form tracks with regular (even integrable) dynamics.

In what follows the regular boundary components of the flower billiard tables will be sometimes referred to as to petals.

\section{Main Results}

This section presents several statements, which show that structural elliptic flowers always have tracks which could be chaotic (hyperbolic) or non-chaotic. Likewise the dynamics in a core, which is the complement to tracks in the phase space of a billiard, could be chaotic or not. A billiard in ellipse has a core consisting of such orbit which between any two consecutive reflections off the boundary intersect the segment connecting the focuses.

Denote by $M(A)$, $Q(A)$ the phase space and the configuration space (billiard table) built over the base polygon $A$. A core $Co(M(A))$ is the complement to the tracks $Tr(A)$ in the phase space $M(A)$. The projection of the core $Co(M(A))$ to the billiard table $Q(A)$ will be called a base core and denoted by $Co(Q(A))$.

We start with the following simple Proposition.

\begin{Proposition}\label{prop1}
Let $A$ be a regular polygon with $n \leq 6$ sides. Then any structural elliptic flower billiard with the base $A$ will have two tracks. The EF billiards built over triangles and squares will always have a base core $A$. The EF billiards built over pentagons and hexagons will have the same base core $A$ if their boundaries belong to the zones 1 and 2. Otherwise, if the boundary of an elliptic flower billiard table intersect zone 3, then a core is a convex polygon, which is a proper subset of $A$.
\end{Proposition}
\begin{proof}
If the boundary of a billiard table $Q(A)$ does not intersect zone 3, then the statement immediately follows from the construction of the structural multilayered elliptic flowers and from the optical definition of ellipses. 

Let now the boundary of $Q(A)$ intersects zone 3 (see Fig.2(b)). Connect then each point, which is an end point of a boundary component (elliptic arc) in this zone, by the straight segments to the another focus ($D$ in Fig. \ref{fig2} (b)) of the corresponding ellipse. (Observe, that this end point $B$ is already connected to one focus of this ellipse by the straight line which belongs to the boundary of the zone 3). Then the collection of all straight lines, which are going through all such segments, cut out a smaller convex polygon $A'$ from the $A$. Again, the construction of the structured multilayered elliptic flowers, together with the optical definition of the ellipses, ensures that $A'$ is indeed the base core of the billiard in $Q(A)$.
\end{proof}

Actually any structural elliptic flower billiard has two tracks of positive measure. A proof of the corresponding general statement uses the same argument as the Proposition above. 

\begin{Theorem}\label{thm1}
Let $A$ be a convex polygon in the Euclidean plane with the vertices $A_1, A_2,...,A_n$. Consider any billiard table $Q(A)$, which is a structured elliptic flower built over the base polygon $A$. Then the billiard in $Q(A)$ has two positive measure tracks, consisting of orbits moving clockwise and counter-clockwise respectively.  
\end{Theorem}
\begin{proof}
The base polygon $A$ is convex. Therefore it lies on one side of a line going through any side of $A$. Then by construction of the structural elliptic flowers it follows from the optical definition of ellipses that if any link of a billiard orbit goes through some vertex of $A$ then all other links of such orbit also contain one vertex of $A$.

By making now an analogous consideration to the one in the proof of Proposition \ref{prop1}, we obtain that a billiard in any structured elliptic flower $Q(A)$ with the base $A$ has two tracks. Likewise, the corresponding core coincides with $A$, if the boundary $\partial Q(A)$ belongs only to the zones 1 and 2. Otherwise, a core $Co(Q(A))$ of the corresponding billiard is a convex polygon, which is a proper subset of the base polygon $A$.
\end{proof}

\begin{Corollary}\label{cor1}
The core $Co(M(A))$ consists of all such billiard orbits that each their link intersect $A$.
\end{Corollary}

Observe that the Corollary \ref{cor1} implies also that the core is an invariant subset of the corresponding billiard dynamics. 

This statement immediately follows from the construction of the structural multilayered elliptic flowers and the proof of the Proposition \ref{prop1} and Theorem \ref{thm1}.

\begin{Corollary}\label{cor2}
The core $Co(M(A))$ of an elliptic flower billiard has a positive measure in the phase space $M(A)$. That means $\mu(Co(M(A)))>0$.
\end{Corollary}

\begin{Definition}\label{def4}
A core $Co(M(A))$ of an elliptic flower billiard will be called an ergodic core if it is an ergodic component of this dynamical system. Analogously, a track will be called ergodic if it is an ergodic component of the corresponding elliptic flower billiard.
\end{Definition}

As was already mentioned, in what follows we will mainly concentrate on the case when the base polygon $A$ is a regular polygon with $n \leq 6$ sides.

\begin{Proposition}\label{prop2}
If the base polygon $A$ is a regular triangle or a square then the dynamics on the corresponding core of a structural elliptic flower over $A$ cannot be ergodic.
\end{Proposition}
\begin{proof}
Indeed, according to the Remark 1 (see also Fig. \ref{fig1}) the boundary $\partial Q(A)$ of a structural elliptic flower over a regular triangle consists of six regular components, which contain three pairs of arcs of two confocal ellipses. Moreover, a period two orbit which connects the ends of the small axes for any pair of these ellipses is a stable elliptic periodic orbit. Therefore the dynamics on a core of EF billiards, built over a regular triangles, is not ergodic. The same is true for the four pairs of the arcs of confocal ellipses in case when $A$ is a square. It is easy to see that the boundary of a structural elliptic flower built  over square contains two pairs of arcs of confocal ellipses (Fig. \ref{fig1}). Therefore, in this case there are at least two elliptic periodic orbits (of period two) in the core. Therefore the core is not an ergodic (chaotic) components of the corresponding elliptic flower billiard.
\end{proof}

The large and the small axes of ellipses will be denoted by $a_i$ and $b_i$ respectively, where the index $i$ refers to one of the ellipses which contain the $i$-th boundary components of an elliptic flower billiard table. Also $c^2=a^2-b^2$, where $c$ is a half-length of the segment connecting the focuses of an ellipse.

In this paper our main goal is to demonstrate in the simplest examples the richness of dynamics of the elliptic flowers billiards. Therefore in what follows we will consider the simplest billiards in this class called special one layer (SOL) elliptic flowers.

\textbf{Construction of special one layer elliptic flowers:}
Let $A$ be an arbitrary convex polygon with $n$ vertices. Consider the straight semi-lines which have the ends at the centers of the sides of $A$, are orthogonal to the corresponding sides, and do not intersect $A$.  Then a special one layer elliptic flower is a billiard table with $n$ boundary components, which are the arcs of ellipses with the focuses located at the ends of the small diagonals of $A$ and the ends on two semi-lines orthogonal to the (intersecting) sides of $A$, which the corresponding small diagonal connects (Fig. \ref{fig4}). If $A$ is a triangle, then the sides play also the role of small diagonals.   

\begin{figure}[h]
	\begin{center}
		\includegraphics[width=6cm]{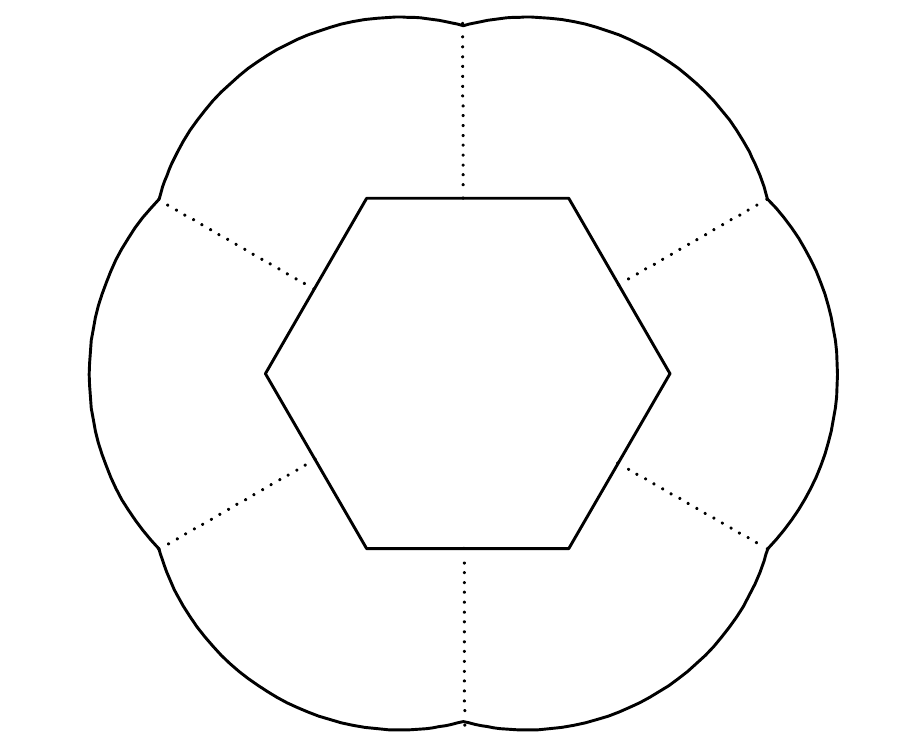}
		\caption{Narcissus: structural one layer elliptic flower over a regular hexagon.}~\label{fig4}
	\end{center}
\end{figure}
The next statement shows that, there are structural elliptic flowers with an ergodic core if a regular base polygon has at least five sides. (It should be mentioned though that structural elliptic flowers built over non-regular triangles and rectangulars may also have ergodic dynamics on their core).  

\begin{Theorem}\label{thm2}
If a number of sides of the base regular polygon $A$ is  $n \geq 5$, then there exist structured elliptic flowers $Q(A)$ such that dynamics of the corresponding billiard, restricted to the core, is hyperbolic and ergodic.  
\end{Theorem}
\begin{proof} For the sake of clarity and brevity we will present a proof only for SOL elliptic flowers. 

Consider a hyperbolic and ergodic circular flower billiard \cite{6}, where all the corresponding circles have the same radius $R$, and the centers of all these circles are at the centers of the small diagonals of $A$.  Recall that all circular arcs, which complement regular components of an ergodic circular flower to the entire circles, belong to the interior of the billiard table. 

Recall now that the maximal radius of curvature of the boundary of an ellipse with the parameters (semi-axes lengths) $a$ and $b$ equals
\begin{equation}\label{eq1}
    \frac{a^2}{b}=b+\frac{c^2}{b}                 
\end{equation}
        
The osculating circle with this maximal radius is tangent to an ellipse at an end of its small axis. We call such circle a maximal osculating circle.         
        
Substitute now each of these $n$ circles by the tangent to them ellipses with the focuses located at the centers of the corresponding shortest diagonals of $A$ and the small axis orthogonal to these diagonals. As a result of we obtain a special one layer elliptic flower $Q(A)$, where the ends of the regular components of the boundary $\partial Q(A)$ lie on the lines orthogonal to the sides of the regular polygon $A$ at their centers.

All $n$ regular components of the boundary of $Q(A)$ are the arcs of ellipses, each of which has focuses at the ends of the shortest diagonals $A_i,A_{(i+2)}$ of $A$. In what follows we will refer to such diagonals of any polygon as to small diagonals (just having in mind that in a general convex polygon not all small diagonals could be also the shortest ones).
Observe that there are no ellipses among them with the same focuses because $n \geq 5$.

Denote by $l$ the length of sides of the base polygon $A$. Then the lengths of its small diagonals are equal to  $2l\cos(\pi/n)$. So $c$ in the relation (\ref{eq1}) equals $l\cos(\pi/n)$. The centers of all ellipses are at the centers of the corresponding shortest (small) diagonals of the basic polygon $A$. The distances between these points and the center of the polygon $A$ equal 
\begin{equation}\label{eq2}
 \frac{l(\cos(2\pi/n))}{2\sin(\pi/n)}
\end{equation}
 
Consider now the maximal osculating circle tangent to a boundary component at its point with the minimal curvature (i.e. at the end of the  small axis of the corresponding ellipse). Take sufficiently large radius $R$ for an ergodic circular flower over $A$. Observe that $R=b$, where $b$ is the length of the small semi-axis of the ellipse tangent to a regular component of the boundary of the circular flower.  It immediately follows from the relations (\ref{eq1}) and (\ref{eq2}) that for any $n\geq 5$ the maximal osculating circle consists of two arcs. The first arc goes outside of the elliptic flower billiard table, and the second arc completely lies within the billiard table $Q(A)$ if $R$ is large enough or if $c$ is small enough. In other words, both these conditions mean that the base polygon should be sufficiently small in comparison with the size of the elliptic flower billiard table $Q(A)$.
 
Take now any interior point $q$ of any regular component of the boundary $\partial Q(A)$. Then any straight segment, which connects $q$ to any other regular component of the boundary will intersect the osculating circle at the point $q$ within the billiard table $Q(A)$.

Sufficient conditions of ergodicity of the two-dimensional billiards were extensively studied (see e.g. \cite{8,13}). However, the cores of elliptic flower billiards have the special property which allows to reduce these conditions just to one. Namely, any orbit in the core cannot have two consecutive reflections off one and the same regular component (an elliptic arc) of the boundary $\partial Q(A)$. Therefore, it is enough for ergodicity that any link of any orbit in the core intersect both osculating circles of two ellipses at the points, which this link connect.

Recall now that any orbit of the core of the elliptic flower billiard intersects the base 
polygon $A$ between any two consecutive reflections off the boundary. Thus each link of such orbit intersect both osculating circles tangent to the boundary $\partial Q(A)$ at the ends of this link.  
 
Therefore the defocusing mechanism \cite{3,4,6} ensures that dynamics in the core of the elliptic flower $Q(A)$ is completely hyperbolic, i.e. it has nonvanishing Lyapunov exponents. The ergodicity of the dynamics in the core immediately follows from \cite{8,13}.
\end{proof}

{\bf Remark 3.} Observe that the chaotic cores of the special one layer elliptic flowers are nonempty subsets of the base polygon $A$. The sizes of these subsets decrease when $b=R$ increases. These cores are regular convex polygons, which, when $b$ tends to infinity, converge to the polygon formed by the lines orthogonal to the sides of $A$ at their ends. 

{\bf Remark 4.} A necessary condition of ergodicity of a two-dimensional billiard says that all focusing components of the boundary should be absolutely focusing \cite{12}. However, all orbits of the core cannot have two consecutive collisions with one and the same regular component of the boundary. Therefore the absolute focusing condition does not play any role here. Observe also that in the papers dealing with general results on ergodicity of billiards only ergodicity in the entire phase space was studied, rather than ergodicity of dynamics on some subset of the phase space. However, the corresponding proofs can be applied word by word for our case, when the core is an ergodic component with a positive, but not full, measure.

We will consider now conditions of hyperbolicity and ergodicity of dynamics in tracks. The orbits in tracks, in difference with the core, may have any number of consecutive reflections off the one and the same regular boundary component of an elliptic flower. Therefore, in dealing with tracks we will require all regular components of the boundary to be absolutely focusing. Recall \cite{7,9} that a $C^2$-smooth focusing curve $\Gamma$ is absolutely focusing if any parallel beam of rays, which is fallen on $\Gamma$, will leave this curve after any series of consecutive reflections off $\Gamma$ as a convergent (focused) beam.

\begin{Theorem}\label{thm3}
There exist elliptic flower billiards with hyperbolic and ergodic tracks.
\end{Theorem}
\begin{proof} In what follows tracks, which are hyperbolic and ergodic, will be called chaotic tracks. Again for the sake of clarity and simplicity we will consider only special one layer flower billiards.

Consider now (the same as above) small perturbations of the hyperbolic and ergodic circular flower, when all $n$ equal circles become equal ellipses with  focuses at the ends of the small diagonals of $A$. Let also their small axis equals $2R$, and thus $b=R$.  All $n$ regular components of the boundary of $Q(A)$ are the arcs of ellipses, each of which has focuses at the ends of the shortest diagonals $A_i,A_{(i+2)}$ of $A$.  Recall that there are no ellipses with the same focuses because $n \geq 5$.
 
The polygon $A$ is regular and the corresponding circular flower has identical arcs of a circle with radius $R$. Therefore each of the circles containing a component of this circular flower intersects only two neighboring boundary components. We can keep this property for an elliptic flower by choosing sufficiently large (with respect to the size of the base polygon $A$) circles. Indeed, it immediately follows from the relation (\ref{eq1}).
 
Now, it follows from the symmetry of $A$ and the symmetry of our construction that the maximal osculating circles to the ellipses, which form the boundary of elliptic flower billiard table intersect only two neighboring osculating circles. Besides any largest osculating circle does not contain the ends (which belong to the boundary of the billiard table) of the short diagonals of the neighboring osculating circles. Therefore divergence of the orbits beats convergence at any link of an orbit which connects two points of different boundary components.
 
However in tracks, in difference with core, orbits may have any number of consecutive reflections off one and the same boundary components. Therefore each regular boundary component should be absolutely focusing. The condition for the symmetric arcs of ellipses with center at an end of the small diagonal of ellipse to be absolutely focusing was found in \cite{21}.
This condition says that projections of these arcs to the large axis of ellipse should not exceed $a/\sqrt{2}$. It is easy to see that this condition is satisfied if we choose sufficiently large $R=b$ (for a fixed base polygon $A$ with at least five sides). Clearly for regular triangles and for the squares this condition never holds. 
 
A sufficient condition of the absolute focusing for petals of SOL elliptic flowers can be obtained from an elementary geometric consideration. This sufficient condition of absolute focusing says that the angle between the minor axis of an ellipse and the line connecting the center of an ellipse to the end point of the corresponding component of the boundary (petal) must be smaller than $\pi/4$. Clearly this angle is a decreasing function of the length $b$ of the small axis, and the condition of the absolute focusing holds for sufficiently large $R=b$.    
\end{proof}

By combining now the proofs of theorems and \ref{thm1} and \ref{thm2} we obtain 

\begin{Theorem}\label{thm4}
There exist hyperbolic elliptic flowers billiards with phase space consisting of three ergodic components of positive measure, which are two chaotic tracks and a chaotic core. 
\end{Theorem}
Exactly the same proof, as the one of Theorem \ref{thm3} goes again for special one layer elliptic flowers with at least five sided regular base polygon by perturbing a hyperbolic and ergodic circular flower with large circles.

For general multilayered elliptic flowers a proof is quite analogous in logic but contains various extra trigonometric formulas. Once again, our goal in this paper is just to show that the dynamics described in theorems \ref{thm2}, \ref{thm3} and \ref{thm4} does exist in elliptic flowers. General multilayered elliptic flowers provide for more possibilities, some of which will be described in the last section. 

\textbf{Wild Rose and Narcissus: the simplest elliptic flower billiards with chaotic tracks and chaotic core.}
It was shown above that, if an elliptic flower has chaotic tracks and chaotic core, then a minimal number of sides of a basic polygon $A$ is five. Besides, a general multilayered elliptic flower billiard table, built over a 5-polygon, has more than five petals (regular components of the boundary). We give now the simplest example of an elliptic flower with all the properties considered in the theorems \ref{thm2}, \ref{thm3} and \ref{thm4}. (It can be done by choosing $b$ either large or small enough).   

Consider a regular convex pentagon $A$. Construct now a special one layer (SOL) elliptic flower over $A$. This flower has five petals, and therefore it can be naturally called a wild rose elliptic flower billiard. Let also each ellipse containing these petals have focuses at the ends of the corresponding small diagonal of $A$. It was shown in Theorem \ref{thm4} that this billiard has a chaotic core surrounded by chaotic tracks if the size of the base pentagon $A$ is relatively small in comparison with the size of the corresponding billiard table $Q(A)$.  
This wild rose billiard is the simplest one to built in a lab. Observe also that an experimenter can choose a relevant size of this devise.
Another simple example is the Narcissus billiard, which is built over a regular hexagons exactly at the same fashion as the wild rose was built over a regular pentagon (Fig. \ref{fig4}).  

\section{Some Open Questions and Conjectures}

As it was shown above, elliptic flower billiards demonstrate a large variety of behaviors, some of which have never been observed before, or even conjectured to exist. Therefore the first natural general question for the future studies of elliptic flowers is to understand how the changes in the dynamics of these billiards occur. Elliptic flower billiard can be specified by several parameters, all of which have a clear and visual geometric meaning. Therefore various types of bifurcations, which occur in these billiards, could be described in some clear geometric terms.

Certainly, some trivial answers are immediately there, e.g. how a core may loose ergodicity by acquiring  elliptic periodic orbit(s). However, a question on appearance of ``vortices" with regular dynamics within chaotic flows in tracks is, certainly, of interest. Another question is whether immediately infinitely many vortices appear in chaotic tracks as a result of such transition. A general intuition suggests that it is likely the case. However, even in this situation it would be interesting to look at the structure of the set of vortices, and determine whether or not there is some kind of  ``coherence" between them.

Another interesting issue, although probably essentially a technical one, is about construction of elliptic flowers with some specific properties for any convex polygon taken as a base. The results of the present paper were intentionally dealing with the simplest situation, when the base polygon is regular.

There are still, at least theoretical, possibilities that there exist elliptic flowers billiards, which have some interesting properties not considered in the present paper. 
Recall  that in all the examples of multilayered elliptic flower billiards, studied in the present paper, all petals (regular components of the boundary) belonged only to the first or to the second layer of the base polygon. 

Moreover, mostly even the simpler class of special one layer elliptic flowers was analysed. One may expect that the studies of elliptic flowers with the higher levels petals will bring some new examples of an unseen before behaviors, which are amenable to rigorous studies.  

Observe also that all elliptic flowers considered in our paper had $C^0$ boundary. However, it is easy to construct elliptic flowers with $C^1$ and with $C^2$ boundaries. (See, for instance, the string construction described in the section 2). Could such elliptic flowers billiards demonstrate analogous properties to the ones studied here? Certainly the tracks will be there, but it is not so clear about other features. 

Intriguing could be also the future studies of quantum elliptic flowers billiards. One may expect to see some new surprises there.

Another general question is whether it is possible to construct billiards with similar dynamics when a core is any convex set on a plane rather than a polygon.

\textit{Acknowledgements}. I am indebted to Hassan Attarchi for help with figures and useful comments.

\end{document}